\documentclass[reqno]{amsart}
\usepackage{amsmath}
\usepackage[active]{srcltx}
\usepackage{amssymb}
\usepackage[dvips]{graphicx}
\newtheorem{Proposition}{Proposition}[section]
\newtheorem{Corollary}{Corollary}[section]
\newtheorem{Lemma}{Lemma}[section]
\newtheorem{Theorem}{Theorem}[section]
 \newtheorem*{theorem*}{Theorem}

\theoremstyle{definition}
\newtheorem{Definition}{Definition}[section]

\theoremstyle{definition}
\newtheorem{Remark}{Remark}

\usepackage{enumerate}

\begin{document}

\title{Uniform Distribution of Prime Powers and sets of Recurrence and van der Corput
  sets in $\mathbb{Z}^k$} 
\author[V. Bergelson]{Vitaly Bergelson}
\thanks{The first author gratefully acknowledges the support of the NSF under grant DMS-1162073.}
\address[V. Bergelson]{Department of Mathematics\\ Ohio State University \\ Columbus, Ohio 43210, USA}
\email{vitaly@math.ohio-state.edu}

\author[G. Kolesnik]{Grigori Kolesnik}
\address[G. Kolesnik]{Department of Mathematics\\ California State University \\ Los Angeles, CA 90032, USA}
\email{gkolesnik@sbcglobal.net}

\author[M. Madritsch]{Manfred Madritsch}
\address[M. Madritsch]{Department for Analysis and Computational Number
  Theory\\Graz University of Technology\\A-8010 Graz, Austria}
\email{madritsch@math.tugraz.at}

\author[Y. Son]{Younghwan Son}
\address[Y. Son]{Department of Mathematics\\ Ohio State University \\ Columbus Ohio 43210, USA}
\email{son@math.ohio-state.edu}

\author[R. Tichy]{Robert Tichy}
\thanks{The last author gratefully acknowledges the support of the FWF under grant P26114.}
\address[R. Tichy]{Department for Analysis and Computational Number
  Theory\\Graz University of Technology\\A-8010 Graz, Austria}
\email{tichy@tugraz.at}

\bigskip

\gdef\shorttitle{ }
\maketitle

\setcounter{section}{0}

\begin{abstract} 
 We establish new results on sets of recurrence and van der Corput sets in $\mathbb{Z}^k$ which refine and unify some of the previous results obtained by S\'{a}rk{\H o}zy, Furstenberg, Kamae and M\`endes France, and Bergelson and Lesigne. The proofs utilize a general equidistribution result involving prime powers which is of independent interest. 
 \end{abstract}


\section{Introduction}

 A. S\'{a}rk{\H o}zy established in \cite{Sa1}, \cite{Sa2} and \cite{Sa3} the following surprising results:
\begin{Theorem}
\label{sarkozy}
Let $E \subset \mathbb{N}$ be a set of positive upper density:
$$\overline{d}(E) := \limsup_{N \rightarrow \infty} \frac{|E \cap \{1, 2, \cdots, N \} |}{N} > 0.$$
\begin{enumerate}[(i)]
\item Let $k \in \mathbb{N} = \{ 1, 2, 3, \dots \}$. Then one can find arbitrarily large $n \in \mathbb{N}$ such that for some $x, y \in E$, $x-y = n^k$.

\item Denote by $\mathcal{P}$ be the set of prime numbers $\{2, 3, 5, 7, 11, \cdots \}$. One can find arbitrarily large $p \in \mathcal{P}$ such that for some $x,y \in E$, $x - y = p-1$. Also one can find arbitrarily large $q \in \mathcal{P}$ such that $x-y = q+1$.
\end{enumerate}
\end{Theorem}

\begin{Remark} \mbox{}
\begin{enumerate}
\item In \cite{Sa1} the case of the equation $x-y=n^2$ is considered and a quantitative
  refinement of statement $(i)$ is proved by an application of the
  Hardy-Littlewood method. Let $A(N)=\lvert E\cap\{1,\ldots,N\}\rvert$ and
  assume that the difference set of $E$ does not contain a square of an
  integer. It is proved in \cite{Sa1} that
  \[
    \frac{A(N)}{N}=O\left(\frac{(\log\log N)^{\frac23}}{(\log N)^{\frac13}}\right)=o(1),
  \]
  which implies assertion $(i)$ of Theorem 1.1. In \cite{Sa2} a lower bound for
  $A(N)$ is established and in \cite{Sa3} similar results are given for
  $n^k$, $k\in\mathbb{N}$, as well as a quantitative version of
  assertion $(ii)$ of Theorem 1.1.

  The best bound on square differences is by Pintz, Steiger and
  Szemer{\'e}di~\cite{PSS}. In particular, they combined the
  Hardy-Littlewood method with a combinatorial construction in order
  to show that
  \[\frac{A(N)}{N}=O\left((\log N)^{-c_n}\right),\]
  where $c_n\to\infty$.

\item It is not hard to see that only shifts by 1 or -1 can ``work" for part (ii) of Theorem \ref{sarkozy}: for any $h \ne \pm 1$ there exists $a$ and $b$ such that the set $a \mathbb{N} + b$ provides a counter example. 
\end{enumerate}
\end{Remark}

Theorem \ref{sarkozy} can also be obtained with the help of the ergodic method introduced by H. Furstenberg in \cite{F1}. While the ergodic method does not provide sharp finitistic bounds, it allows us to see S\'{a}rk{\H o}zy's results as statements about recurrence in measure preserving systems and leads to a variety of strong extensions of Theorem \ref{sarkozy}.

To illustrate how the ergodic method works, let us consider, for example, the following polynomial refinement, due to Furstenberg, of the classical Poincar\'e recurrence theorem.

\begin{Theorem}[\cite{F2}, Theorem 3.16]
\label{fur-sar}
Let $(X, \mathcal{B}, \mu)$ be a probability space and let $T$ be an invertible measure preserving transformation.\footnote{ We will refer to the quadruple $(X, \mathcal{B}, \mu, T)$ as a measure preserving system and will tacitly assume that  $T$ is invertible and $\mu(X) = 1$.} Let $A \in \mathcal{B}$ with $\mu(A) > 0$.
 For any $g(t) \in \mathbb{Z}[t]$ with $g(0) = 0$, there are arbitrarily large $n \in \mathbb{N}$ such that $\mu(A \cap T^{-g(n)}A) > 0.$
\end{Theorem}

Theorem \ref{fur-sar} implies the following combinatorial result which generalizes Theorem \ref{sarkozy} ($i$).

\begin{Theorem}[\cite{F2}. Proposition 3.19]
\label{sarkozy 2}
Let $E \subset \mathbb{N}$ have positive upper Banach density:
$${d}^*(E) := \limsup_{N -M \rightarrow \infty} \frac{|E \cap \{M, M+1, \cdots, N-1 \} |}{N-M} > 0.$$ 
For any $g(t) \in \mathbb{Z}[t]$ with $g(0)=0$, there are arbitrarily large $n$ such that
$${d}^* (E \cap (E- g(n)) ) > 0.$$
\end{Theorem}



To derive Theorem \ref{sarkozy 2} from Theorem \ref{fur-sar} one can utilize Furstenberg's correspondence principle (see \cite{B3}), which for the case in question says that for any $E \subset \mathbb{N}$ with $d^*(E)>0$ there exist an invertible measure preserving system $(X, \mathcal{B}, \mu, T)$ and $A \in \mathcal{B} $ with $\mu(A) = d^*(E)$ such that for any $n \in \mathbb{Z}$ one has
$$d^*(E \cap E-n) \geq \mu(A \cap T^{-n}A).$$

One can also show that Theorem \ref{sarkozy 2} implies Theorem \ref{fur-sar}. To see this one can utilize Theorem \ref{sarkozy} from \cite{B1} (see also \cite{BM}.)

\begin{Definition}
\label{def1.1}
A set $R \subset \mathbb{N}$ is called {\it{a set of recurrence}} if for any invertible measure preserving system $(X, \mathcal{B}, \mu , T)$ and any $A \in \mathcal{B}$ with $\mu(A) > 0$, there exists $n \in R$ such that $\mu(A \cap T^{-n}A) > 0$.
\end{Definition}


Applications of ergodic theory to combinatorics and number theory bring to life various natural refinements of Definition \ref{def1.1}. Here is a sample of some notions of recurrence relevant to this paper.\footnote{For convenience of the discussion we define these notions for $\mathbb{N}$. We will introduce later the more general notions in $\mathbb{Z}^k$.}
\begin{itemize}
\item {\it{Nice recurrence}} (See \cite{B2}). A set $R \subset \mathbb{N}$ is called a set of nice recurrence if for any measure preserving system $(X, \mathcal{B}, \mu , T)$, any $A \in \mathcal{B}$ with $\mu(A) > 0$, and $\epsilon > 0$, there exist infinitely many $n \in R$ such that $\mu(A \cap T^{-n}A) \geq \mu(A)^2 - \epsilon$.
\item {\it{vdC sets}} (See \cite{KM}). A set $H \subset \mathbb{N}$ is called a van der Corput set, or a vdC set if the uniform distribution mod 1 of the sequence $(x_{n+h} - x_n)_{n \in \mathbb{N}} $ for any $h \in H$ implies the uniform distribution mod 1 of the sequence $(x_n)_{n \in \mathbb{N}}.$  Equivalently (see \cite{BL}), $H \subset \mathbb{N}$ is a vdC set if for any sequence of complex numbers $(u_n)_{n \in \mathbb{N}}$ of modulus $1$, such that for any $h \in H$ $\lim\limits_{N \rightarrow \infty} \frac{1}{N} \sum\limits_{n=1}^N u_{n+h} \overline{u_n} = 0$, one has  $\lim\limits_{N \rightarrow \infty} \frac{1}{N} \sum\limits_{n=1}^N u_n = 0$. 
\end{itemize} 
Clearly any set of nice recurrence is a set of recurrence. It is somewhat less obvious that any vdC set is a set of recurrence. (See \cite{KM} for the proof.) One can also show that not every set of recurrence is a set of nice recurrence (see \cite{Mc}) and that not every set of recurrence is a vdC set (see \cite{Bou}.)

It turns out that the sets mentioned above, namely the sets $\mathcal{P} -1$, $\mathcal{P}+1$ as well as the sets of the form $\{g(n) : n \in \mathbb{Z} \}$, where $g(t) \in \mathbb{Z}[t]$ and $g(0) =0$, are sets of nice recurrence and also $vdC$ sets. (See, for example, \cite{BFMc} and \cite{BL}.)

As a matter of fact the following simultaneous extension of Theorem \ref{sarkozy} and Theorem \ref{fur-sar}
holds true. (See Proposition 1.22 and Corollary 2.13 in \cite{BL}. See also Theorem \ref{sarkozy type} below.)
\begin{Theorem}
\label{sar extension}
For any $g(t) \in \mathbb{Z}[t]$ with $g(0) = 0$, the sets $\{g(p-1) : p \in \mathcal{P} \}$ and $\{g(p+1) : p \in \mathcal{P} \}$ are sets of nice recurrence and also are vdC sets.
\end{Theorem}

One of the goals of this paper is to obtain a number of $n$-dimensional refinements and generalizations of Theorem \ref{sar extension}.

Our proofs of the results on sets of (nice) recurrence and (various enhanced versions of) van der Corput sets rely on the following general result about uniform distribution, which is of independent interest.

\newtheorem*{ud}{Theorem \ref{ud}}
 \begin{ud}[see Section 2]
Let $\xi(x) = \sum_{j=1}^{m} \alpha_j x^{\theta_j}$, where $0 <\theta_1 <
\theta_2 < \cdots < \theta_m$, $\alpha_j$ are non-zero reals and assume that if all
$\theta_j \in \mathbb{Z}^+$, then at least one $\alpha_j$ is irrational. Then
the sequence $ (\xi(p))_{p \in \mathcal{P}}$ is u.d. mod $1$.\footnote{We are tacitly assuming that the set $\mathcal{P} = (p_n)_{n \in \mathbb{N}}$ is naturally ordered, so that $(f(p))_{p \in \mathcal{P}}$ is just another way of writing $(f(p_n))_{n \in \mathbb{N}}$.} 
\end{ud}

One of the applications of Theorem \ref{ud} is the following von Neumann-type theorem along primes.

\newtheorem*{ergodic}{Theorem \ref{ergodic}}
\begin{ergodic}[see Section 3]
Let $c_1, \dots , c_k$ be distinct positive real numbers such that $c_i \notin \mathbb{N}$ for $i=1,2, \dots, k$. Let $U_1, \dots , U_k$ be commuting unitary operators on a Hilbert space $\mathcal{H}$.  Then,
$$\lim_{N \rightarrow \infty} \frac{1}{N}  \sum_{n=1}^N U_1^{[p_n^{c_1}]} \cdots U_k^{[p_n^{c_k}]} f = f^*,$$
where $p_n$ denotes $n$-th prime and $f^*$ is the projection of $f$ on $\mathcal{H}_{inv} (:= \{ f \in \mathcal{H} :  U_i f = f \,\, \textrm{for all} \,\, i\})$.
 \end{ergodic}

Theorem \ref{ergodic}, in turn, has the following corollaries.

\newtheorem*{ergodic sequence}{Corollary \ref{ergodic sequence}}
\begin{ergodic sequence}
Let $c_1, c_2, \dots , c_k$ be positive, non-integers. Let $T_1, T_2, \dots , T_k$ be commuting, invertible measure preserving transformations on a probability space $(X, \mathcal{B}, \mu)$. 
Then, for any $A \in \mathcal{B}$ with $\mu(A) > 0$, one has 
$$\lim_{N \rightarrow \infty} \frac{1}{N} \sum_{n=1}^{ N} \mu(A \cap T_1^{-[p_{n}^{c_1}]} \cdots T_k^{-[p_n^{c_k}]} A) \geq \mu^2(A),$$
where $p_n$ denotes the $n$-th prime.
\end{ergodic sequence}

\newtheorem*{prop}{Corollary \ref{prop}}
\begin{prop}
Let $c_1, \cdots ,c_k$ be positive non-integers.
If $E \subset \mathbb{Z}^k$ with ${d^*}(E) > 0$, then there exists a prime $p$ such that $([p^{c_1}], \cdots , [p^{c_k}] ) \in E - E$. 
Moreover,
$$\liminf_{N \rightarrow \infty} \frac{ | \{ p \leq N:   ([p^{c_1}], \cdots , [p^{c_k}] ) \in E - E\} | }{ \pi(N) } \geq {d^*}(E)^2,$$
where $\pi(N)$ is the number of primes less than or equal to $N$.
\end{prop}

Before formulating additional results to be proved in this paper we have to introduce some pertinent definitions. (A detailed discussion of various additional notions of sets of recurrence in $\mathbb{Z}^k$ is provided in Section 4.)

\begin{Definition} 
A set $D$  $ \subset \mathbb{Z}^k$ is {\it{a set of nice recurrence}} if given any ergodic measure preserving $\mathbb{Z}^k$-action $T=(T^{\bold{m}})_{(\bold{m} \in \mathbb{Z}^k)}$ on a probability space $(X, \mathcal{B}, \mu)$, any set $A \in \mathcal{B}$ with $\mu(A) > 0$ and any $\epsilon >0$, we have 
$$\mu(A \cap T^{-\bold{d}} A) \geq \mu^2(A) - \epsilon $$
for infinitely many $\bold{d} \in D$.
\end{Definition} 

\begin{Definition}[cf.\cite{BL}, Definition 1.2.1] 
A subset $D$ of $\mathbb{Z}^k \backslash \{0\}$ is {\it{a van der Corput set}} ({\it{vdC set}}) if for any family $(u_{\bold{n}})_{\bold{n} \in \mathbb{Z}^k}$ of complex numbers of modulus $1$ such that 
$$\forall \bold{d} \in D, \,\,  \lim_{N_1, \cdots, N_k \rightarrow \infty} \frac{1}{N_1 \cdots N_k} \sum_{\bold{n} \in \prod_{i=1}^k [0, N_i)} u_{\bold{n}+\bold{d}} \overline{u_{\bold{n}}} = 0$$
we have $$\lim_{N_1, \cdots, N_k \rightarrow \infty} \frac{1}{N_1 \cdots N_k } \sum_{\bold{n} \in \prod_{i=1}^k [0, N_i)} u_{\bold{n}} = 0.$$
 \end{Definition}

The following results are obtained in Sections 4 and 5.

\begin{Theorem}[cf. Theorem \ref{sarkozy type}]
If $\alpha_i$ are positive integers and $\beta_i$ are positive and non-integers, then
$$ D_1 = \{ \left(  (p-1)^{\alpha_1}, \cdots , (p-1)^{\alpha_k}, [(p-1)^{\beta_1}], \cdots , [(p-1)^{\beta_l}]  \right) : \, p \in\mathcal{P} \},$$
and
$$ D_2 = \{ \left(  (p+1)^{\alpha_1}, \cdots , (p+1)^{\alpha_k}, [(p+1)^{\beta_1}], \cdots , [(p+1)^{\beta_l}]  \right) : \, p \in \mathcal{P} \}$$
are vdC sets and also sets of nice recurrence in $\mathbb{Z}^{k+l}$.
\end{Theorem}

\newtheorem*{semi ergodic cor}{Corollary \ref{semi ergodic cor}}
\begin{semi ergodic cor}[see Section 5]
Let $D_1$ and $D_2$ be as in Theorem \ref{sarkozy type}.
If $E \subset \mathbb{Z}^{k+l}$ with ${d^*}(E) > 0$, then for any $\epsilon > 0$
\begin{equation*}
 \{ \bold{d} \in D_i : d^*(E \cap E - \bold{d} ) \geq d^*(E)^2 - \epsilon \} 
 \end{equation*}
 has positive lower relative density\footnote{ For sets $A \subset B \subset \mathbb{Z}^m$, the lower relative density of $A$ with respect to $B$ is defined as $$ \liminf_{n \rightarrow \infty} \frac{|A \cap [-n,n]^m|}{|B \cap [-n,n]^m|}.$$ } in $D_i$ for $i=1, 2$.
 Furthermore,
 $$\liminf_{N \rightarrow \infty} \frac{\left| \{ p \leq N :\left(  (p - 1)^{\alpha_1}, \cdots , (p - 1)^{\alpha_k}, [(p - 1)^{\beta_1}], \cdots , [(p - 1)^{\beta_l}]  \right) \in E - E \} \right| }{\pi(N)} > 0,$$
and
  $$\liminf_{N \rightarrow \infty} \frac{\left| \{ p \leq N :\left(  (p + 1)^{\alpha_1}, \cdots , (p + 1)^{\alpha_k}, [(p + 1)^{\beta_1}], \cdots , [(p + 1)^{\beta_l}]  \right) \in E - E \} \right| }{\pi(N)} > 0.$$
\end{semi ergodic cor}

\section{equidistribution}
\label{sec : equidistribution}
The goal of this section is to prove the following simultaneous extension of the results in \cite{Rh} and \cite{ST} and to derive from it some useful corollaries.
\begin{Theorem}
\label{ud}
Let $\xi(x) = \sum_{j=1}^{m} \alpha_j x^{\theta_j}$, where $0 <\theta_1 <
\theta_2 < \cdots < \theta_m$, $\alpha_j$ are non-zero reals and assume that if all
$\theta_j \in \mathbb{Z}^+$, then at least one $\alpha_j$ is irrational. Then
the sequence $ ( \xi(p) )_{p \in \mathcal{P}} $ is u.d. mod $1$. 
\end{Theorem}

The following notation will be used throughout this paper.
\begin{enumerate}
\item $e(x) = \exp (2 \pi i x)$.
\item  $X \ll Y$ (or $X= O(Y)$) means $ X \leq C Y $ for some positive constant $C$.
\item  $X \asymp Y$ for $C_1 X \leq Y \leq C_2 X$ for some positive constants $C_1, C_2$.
\item $f(x) \lll A$ means that for any $\epsilon > 0$, there is a positive constant $C(\epsilon)$ such that 
$$ | f(x) | \leq C(\epsilon) A x^{\epsilon}.$$
\item $\sum_{p\leq N}$ denotes the sum over primes.
\item The von Mangoldt function is defined as
$$\Lambda(n) = \begin{cases}
  \log p & \text{if $n = p^k$ for some prime $p$ and integer $k \geq 1$} \\
  0 & \text{otherwise} 
\end{cases}$$
\item For $s \in \mathbb{N}$, the $s$-fold divisor function is defined as
\[
d_s(n)=\sum_{n_1\cdots n_s=n}1,
\]
where the sum is extended over all products with $s$ factors.
\end{enumerate}

Before giving the proof of Theorem \ref{ud} we formulate some necessary auxiliary results. We start with the classical Weyl - van der Corput inequality.
\begin{Lemma}[{cf. \cite[Lemma 2.7]{GK}}]
\label{lemma 1}
Let $k$ be a positive integer and $K = 2^k$. Assume that $X, X_1 \in
\mathbb{N}$ and $X < X_1 < 2X$. For any positive $H_1, \cdots , H_k \ll X_1 -
X$ and 
\[S = \left| \sum\limits_{X \leq x \leq X_1 } e(f(x)) \right|\]
we have
\begin{multline*}
\left( \frac{S}{X_1 -X}\right)^K \leq 8^{K-1} \left\{  \frac{1}{H_1^{K/2}} +
  \frac{1}{H_2^{K/4}} + \cdots + \frac{1}{H_k}\right.\\
\left.+ \frac{1}{H_1 \cdots H_k (X_1 -
    X)} \sum_{h_1 = 1}^{H_1} \cdots \sum_{h_k = 1}^{H_k}  \left|  \sum_{x \in
      I(\underline{h})} e(f_1(x)) \right| \right\},
\end{multline*}
where $f_1(x) :=
f(\underline{h}, x) = h_1 \cdots h_k \int_0^1 \cdots \int_0^1
\frac{\partial^k}{\partial x^k} f(x+ \underline{h} \cdot \underline{t} ) \, d
\underline{t}$, $\underline{h} = (h_1, \cdots , h_k)$, $\underline{t} = (t_1,
\cdots , t_k)$ and $I(\underline{h}) = (X, X_1 - h_1 - \cdots - h_k]$.  
\end{Lemma} 

The next lemma provides a useful estimate for polynomial-like functions.
\begin{Lemma}[{\cite[Theorem 2.9]{GK}}]
\label{lemma 2}
Let $q \geq 0$ be an integer and $X \in \mathbb{N}$. Suppose that $f(x)$ has
$(q+2)$ continuous derivatives on an interval $I \subset (X, 2X]$. Assume also
that there is some constant $G$ such that $|f^{(r)}(x)| \asymp G X^{-r}$ for
$r=1, \dots , q+2$. Then $$ S := \left| \sum_{x \in I} e(f(x)) \right| \ll
G^{\frac{1}{4Q-2}} X^{1- \frac{q+2}{4Q-2}} + \frac{X}{G},$$ where $Q=2^q$ and
the implied constant in $\ll$ depends only on $q$ and on the implied constants
in $\asymp$.
\end{Lemma}

We will also need the following estimate involving the von Mangoldt function, the proof of which is based on an identity of Vaughan's type. 

\begin{Lemma}
\label{lemma 3}
Assume $F(x)$ to be any function defined on the real line, supported on $[N/2, N]$
and bounded by $F_0$. Let further $U,V,Z$ be any parameters satisfying $3
\leq U < V < Z < N$, $Z \geq 4U^2$, $N \geq 64 Z^2 U$, $V^3 \geq 32 N$ and
$Z-\frac12\in\mathbb{N}$. Then
$$\left| \sum_n \Lambda(n) F(n)  \right| \ll K
\log N + F_0 +  L (\log N)^8 ,$$
where the summation over $n$ is restricted to
the interval $[N/2,N]$, and $K$ and $L$ are defined by
\begin{align*}
K&=\max_M\sum_{m=1}^\infty d_3(m)\left\vert\sum\limits_{Z<n\leq M} F(mn)\right\vert,\\
L&=\sup\sum_{m=1}^\infty d_4(m)\left\vert\sum\limits_{U < n < V} b(n) F(mn)\right\vert,
\end{align*}
where the supremum is taken over all arithmetic functions $b(n)$ satisfying $|b(n)| \leq d_3(n).$
\end{Lemma}

\begin{proof}
The inequality in question can be easily derived from Lemma 2 and Lemma 3 of
Heath-Brown~\cite{Hea}. The reader should be warned that in \cite{Hea} $F$, $U$, $V$ and $Z$
are denoted by $f$, $u$, $v$ and $z$ and our parameters $N$ and $M$ correspond
to $x$ and $N$, respectively. From Lemma 2 in \cite{Hea} (which is of
combinatorial nature) we immediately obtain the representation:
\[ 
\sum_n \Lambda(n) F(n)=
\Sigma_1+\Sigma'_1-\Sigma_2-\Sigma'_2-\Sigma_3-\Sigma'_3,
\]
where the quantities on the right hand side satisfy the following estimates (see Lemma 3
in \cite{Hea}, Equations (7) and (8)):
\begin{align*}
\Sigma_1,\Sigma'_1,\Sigma_2,\Sigma'_2&\ll K\log N,\\
\Sigma_3,\Sigma'_3&\ll F_0+L(\log N)^8.
\end{align*}
Combining these estimates, the triangle inequality immediately yields our Lemma~\ref{lemma 3}.
\end{proof}

Given a sequence $(x_n)_{n=1}^{\infty}$, its \textit{discrepancy} is defined by
$$D_N(x_n) = \sup_I \left|\frac{ \# \{ n \leq N : x_n \in I  \, (\bmod 1) \} }{N} -(b-a)\right|,$$
where the $\sup$ is taken over all intervals $I = [a, b) \subset [0,1)$.

Note that the sequence $(x_n)$ is u.d. $(\bmod 1)$ if and only if $\lim\limits_{N \rightarrow \infty} D_N(x_n) = 0$.

The proof of Theorem \ref{ud} will be achieved by showing that $\lim_{N \rightarrow \infty} D_N (f(p_n)) =0$. In doing so we will be using the following version of Erd\H os-Tur\'an Inequality.
\begin{Lemma}[cf. {\cite[Theorem 1.21]{DT}, \cite[Theorem 2.5]{KN}}]
\label{lemma 4}
For any real sequence $(x_n)_{n=1}^{\infty}$ and any positive integer $N$ and $H \leq N$, one has:
$$D_N(x_n) \ll \frac{1}{H} + \sum_{h=1}^{H} \frac{1}{h} \left| \frac{1}{N} \sum_{n=1}^N e(h x_n) \right|.$$
\end{Lemma}

The following lemma will serve as the central tool in the proof of Theorem \ref{ud}.
 \begin{Lemma}
 \label{lemma 5}
 Let $X,k,q\in \mathbb{N}$ with $k,q\geq 0$ and set $K=2^k$ and $Q= 2^q$. Let $P(x)$ be a polynomial of degree $k$ with real coefficients. Let $f(x)$ be a real $(q+k+2)$ times continuously differentiable function on $[X/2 , X]$ such that $\left| f^{(r)}(x) \right| \asymp F X^{-r}$ $( r = 1, \dots, q+k+2) $.
Then, if $F = o (X^{q+2})$ for $F$ and $X$ large enough, we have
 $$S := \left| \sum_{X/2 < x \leq X} e(f(x) + P(x)) \right| \ll X^{1 - \frac{1}{K}} + X \left( \frac{\log^k X}{F} \right)^{\frac{1}{K}} + X \left( \frac{F}{X^{q+2}} \right)^{\frac{1}{(4KQ-2K)}}.$$
  \end{Lemma}
\begin{proof}
 Using Lemma \ref{lemma 1} with 
 $H_i = \frac{X}{2K}$, we obtain
 $$\left| \frac{S}{X} \right|^K \ll \frac{1}{H_k} + \frac{1}{H_1 \cdots H_k X} \sum_{h_1 = 1}^{H_1} \cdots \sum_{h_k=1}^{H_k} \left| \sum_{x \in I (\underline{h})} e(f_1(x)) \right| ,$$
 where $I (\underline{h}) = ( X/2 , X - h_1 - \cdots - h_k ]$ and
\[f_1(x)
 := f_1 (\underline{h}, x ) = h_1 \cdots h_k \left[ \int_0^1 \cdots
   \int_0^1 \frac{\partial^k}{ \partial x^k} f(x + \underline{h} \cdot
   \underline{t}) d \, \underline{t} + a_k k! \right],
\]where $a_k$ is the
 leading coefficient of $P(x)$. The function $f_1(x)$ satisfies the conditions
 of Lemma \ref{lemma 2} with $G = h_1 \cdots h_k F/ X^k$. Thus its application
 yields
 \begin{eqnarray*}
 \left| \frac{S}{X}\right|^K &\ll& \frac{1}{X} + \frac{1}{H_1 \cdots H_k X} \sum_{h_1 =1}^{H_1} \cdots \sum_{h_k=1}^{H_k} \left( F^{\frac{1}{4Q-2}} X^{1- \frac{q+2}{4Q-2}} + \frac{X^{k+1}}{F h_1 \cdots h_k } \right) \\
 &\ll& \frac{1}{X} + \left( \frac{F}{ X^{q+2}} \right)^{\frac{1}{4Q-2}} + \frac{\log^k X}{F}.
 \end{eqnarray*}
 This proves the Lemma. 
\end{proof}

\begin{Remark}
Using a better choice of parameters $\underline{H} = (H_1, \cdots , H_q)$, we can easily improve the estimate in Lemma \ref{lemma 5} but since our aim is to prove uniform distribution, the obtained estimate will be sufficient.
\end{Remark}

\begin{Proposition}
\label{prop 1}
Let $P(x)$ be a polynomial of degree $k$ and $f(x) = \sum_{j=1}^r d_j x^{\theta_j}$ with $r \geq 1$, $d_r \ne 0$, $d_j$ real, $0 < \theta_1 < \theta_2 < \cdots < \theta_r$  and $\theta_j \notin \mathbb{Z}^{+}$. Assume that $l < \theta_r < l+1$ for some $l$. 
Let $1 \leq |m| \leq N^{1/10}$. Then
$$\left| \sum_{p \leq N} e(m f(p) + m P(p))\right| \lll N^{1- \frac{1}{3K}} + \frac{N }{ (m N^{\theta_r})^{1/K} } + N^{1- \frac{1}{64KL^5 - 4K}} + N^{1-1/10},$$
where $K=2^k$ and $L=2^l$.
\end{Proposition}

\begin{proof}
We split the sum $S$ into $\leq \log N$ subsums of the form $\left| \sum\limits_{X \leq p \leq 2X} e(m f(p) + mP(p))\right|$ with $2X \leq N$ and evaluate a typical one of them. We can obviously assume that $X \geq N^{9/10}$. By using partial summation formula we obtain
\begin{eqnarray*}
S &:=& \left| \sum_{X \leq p \leq 2X} e(m f(p) + m P(p)) \right| \\
&=& \left| \sum_{n} \frac{\Lambda(n)}{\log n} e(m f(n) + mP(n))  \right| + O(\sqrt{X}) \\
&\ll& \frac{1}{\log X} \left| \sum_{n \in I} \Lambda(n) e(mf(n) + mP(n)) \right| +O(\sqrt{X}),
\end{eqnarray*}
where $I$ is a subinterval of $(X, 2X]$. Denote the last sum by $S_1$ and use
Lemma \ref{lemma 3} with $U = \frac{1}{4} X^{1/5}$, $V= 4 X^{1/3}$ and $Z$ the
unique number in $\frac12+\mathbb{N}$, which is closest to $\frac{1}{4}
X^{2/5}$. We obtain
\begin{equation*}
\begin{split}
S_1& \ll 1 + \log X \left| \sum_{x < \frac{2X}{Z}} d_3(x) \sum_{y > Z, \frac{X}{x} < y < \frac{2X}{x}} e(mf(xy) + mP(xy)) \right| \\
&+ \log^8 X \left| \sum_{x} d_4(x) \sum_{U < y < V, \frac{X}{x} < y \leq \frac{2X}{x}} b(y) e(mf(xy) + mP(xy))\right|.
\end{split}
\end{equation*}
Denote the first sum by $S_2$ and the second sum by $S_3$. To evaluate $S_2$, we use Lemma \ref{lemma 5} to estimate, for a fixed $x$, the sum over $y$. Here (denoting $Y = \frac{X}{x}$) we have 
$$\left| \frac{\partial^j f(xy)}{\partial y^j} \right| \asymp  X^{\theta_r } Y^{-j} $$ 
for any $j$. Furthermore for $j\geq 5(l+1)$ we have
$$\left| m \frac{\partial^j f(xy)}{\partial y^j} \right| \ll m X^{\theta_r - \frac{2}{5}j} \ll X^{\frac{1}{10} + l +1 - \frac{2}{5}j} \leq X^{-1/2}, $$
where we have used that $y>Z\gg X^{2/5}$.
Thus an application of Lemma \ref{lemma 5} yields the following estimate:
\begin{eqnarray*}
S_2 &\lll& \sum_{x \leq 2X/Z} X/x \left[ (\frac{x}{X})^{\frac{1}{K}} + (\frac{1}{m X^{\theta_r}})^{\frac{1}{K}} + X^{-\frac{1}{2} \frac{1}{4K \cdot 8L^5 - 2K}} \right] \\
&\lll& X \left(X^{-\frac{2}{5K}} + \frac{1}{(mX^{\theta_r})^{\frac{1}{K}}} + X^{- \frac{1}{64KL^5-4K} } \right). 
\end{eqnarray*}
Now we need to estimate $S_3$:
$$S_3 \lll \sum_{\frac{X}{V} < x \leq \frac{2X}{U}} \left| \sum_{\substack {U < y < V \\ \frac{X}{x} < y \leq \frac{2X}{x}}} b(y) e(m f(xy) + m P(xy) ) \right| . $$
We split the interval $( \frac{X}{V} , \frac{2X}{U} ]$ into $\leq \log X$ subintervals of the form $I = (X_1, 2X_1]$ and take one of them. Denote the corresponding sum by $S_4$ and use Cauchy's inequality :
\begin{equation*} \begin{split} |S_4|^2 &\leq X_1 \sum_{x \in I} \left| \sum_y b(y) e(m f(xy) + m P(xy)) \right|^2\\
&\ll X_1^2 \frac{X}{X_1} + X_1 \left| \sum_{x \in I} \sum_{A < y_1 < y_2 \leq B} b(y_1) \overline{b(y_2)} e(m ( f(xy_1) - f(xy_2) + P(xy_1) - P(xy_2)))  \right|, 
\end{split} \end{equation*}
where $A = \max \{U, \frac{X}{x} \} $ and $B = \min \{U, \frac{2X}{x} \}$. Changing the order of summation, we get  
$$|S_4|^2 \lll X X_1 + X_1 \sum_{y_1, y_2} \left| \sum_x e(m (f(xy_1) -f(xy_2) + P(xy_1) -P(xy_2))\right| .$$
Now we fix $y_1$ and $y_2 \ne y_1$. The function $g(x) := m(f(xy_1) - f(xy_2))$ satisfies the conditions of Lemma \ref{lemma 5}: 
$$|g^{(j)}(x)| \asymp m \frac{| y_1 -y_2 |}{y_1} X^{\theta_r} X_1^{-j} \ll m X^{\theta_r} \left( \frac{X}{V} \right)^{-j} \ll X^{\theta_r + \frac{1}{10} - \frac{2}{3}j} \ll X^{-\frac{1}{2}}$$
if $j \leq 2l+3$. Using Lemma \ref{lemma 5} with $q=2l+3$ we obtain 
\begin{eqnarray*}
|S_4|^2 &\lll& X X_1 + X_1 \sum_{y_1, y_2} \left( X_1^{1- \frac{1}{K}} + X_1 \left( \frac{y_1 \log^k X }{m |y_1 - y_2| X^{\theta_r}}\right)^{1/K} \right) + X_1 X^{-\frac{1}{2} \frac{1}{4K \cdot 2L^2 - 2K}} \\
&\lll& XX_1 + X^{2- \frac{2}{3k}} + X^2 \left( \frac{\log^k X}{m X^{\theta_r}} \right)^{1/K} + X^{2- \frac{1}{16KL^2 - 4K}}.
\end{eqnarray*}
Summing over all the subintervals completes the proof.
 \end{proof}

\begin{Proposition}
\label{prop 2}
Let $P(x)$ and $f(x)$ be as in Proposition \ref{prop 1}. Then the discrepancy of the sequence $(f(p) + P(p))$ satisfies
$$D_N \lll   N^{-\frac{1}{10}} + N^{- \frac{1}{3K}} + N^{- \frac{\theta_r}{K}} + N^{- \frac{1}{64 K L^5}}.$$
Consequently, the sequence $(f(p) + P(p))$ is u.d. $\bmod 1$. 
\end{Proposition}

\begin{proof}
We use Lemma \ref{lemma 4} with $H = N^{1/10}$ and obtain 
$$D_N \ll \frac{1}{H} + \sum_{h=1}^H \frac{1}{h} \left| \frac{1}{N} \sum_{p \leq N} e(hf(p) + h P(p)) \right|.$$   
Applying Proposition \ref{prop 1} we obtain the claimed result :
$$D_N \lll  N^{-\frac{1}{10}} + N^{- \frac{1}{3K}} + N^{- \frac{\theta_r}{K}} + N^{- \frac{1}{64 K L^5}} .$$
\end{proof}

\begin{proof}[Proof of Theorem \ref{ud}]
We will consider two cases.

Assume first that at least one $\theta_j \notin \mathbb{Z}^+$. Then the
function $\xi(x)$ can be rewritten as $f(x) + P(x)$ as in Proposition \ref{prop
  1}, namely, $P(x)$ is a polynomial and $f(x) = \sum_{j=1}^r d_j x^{\theta_j}$ with $r \geq 1$, $d_r \ne 0$, $d_j$ real, $0 < \theta_1 < \cdots < \theta_r $ and $\theta_j \notin \mathbb{Z}^+$, so $(\xi(p))$ is u.d. $(\bmod 1)$.

Now we assume that all $\theta_j \in \mathbb{Z}^{+}$, i.e. $\xi(x)$ is a polynomial and at least one coefficient $\alpha_j$ is irrational. Then $( \xi(p) )$ is u.d. $\bmod 1$ due to the result of Rhin. (See \cite{Rh}.)
\end{proof}

We list now some corollaries of Theorem \ref{ud} 

\begin{Corollary}
\label{lemma p-1}Let $\xi(x) = \sum_{j=1}^m \alpha_j x^{\theta_j} $ be as in Theorem \ref{ud}. Then for any $h \in \mathbb{Z}$ $(\xi(p-h))_{p \in \mathcal{P}}$ is u.d. $\bmod 1$.
\end{Corollary}
\begin{proof}
Note that for $k < \theta < k+1$, where $k$ is a non-negative integer,  there are $a_1, a_2, \cdots , a_k$ and $g(x)$ such that
$$ (x-h)^{\theta} = x^{\theta} + a_1 x^{\theta-1} + \cdots + a_k x^{\theta-k} + g(x)  \,\,\, \,\,\, \textrm{and} \,\, \lim_{x \rightarrow \infty} g(x) = 0. $$
Then $\sum_{j=1}^m \alpha_j (p -h)^{\theta_j} $ can be written as the sum of $\tilde{\xi}(p) + G(p)$, where $\tilde{\xi}(x)$ is the function as in Theorem \ref{ud} and $\lim\limits_{x \rightarrow \infty} G(x) = 0$. So the result follows. 
\end{proof}

The following result follows from Corollary \ref{lemma p-1} via the classical Weyl criterion (see Theorem 6.2 in Chapter 1 of \cite{KN}.)

\begin{Corollary}
\label{ud mod T}
Let $0 <\theta_1 < \theta_2 < \cdots < \theta_m$ and let $\gamma_1, \gamma_2 , \dots , \gamma_m$ be non-zero real numbers such that $\gamma_i \notin \mathbb{Q}$ if $\theta_i \notin \mathbb{N}$. Let $h$ be an integer. Then
$$( (\gamma_1 (p-h)^{\theta_1}, \gamma_2 (p-h)^{\theta_2}, \cdots , \gamma_m (p-h)^{\theta_m}) )_{p \in \mathcal{P}}$$
is u.d. $\bmod 1$ in $\mathbb{T}^m$.
\end{Corollary}

\begin{Corollary}
\label{ud along}
Let $\theta_1, \cdots, \theta_m$ and $\gamma_1, \cdots, \gamma_m$ be as in Corollary \ref{ud mod T}. Let $q$ and $t$ be positive integes such that $(t,q) = 1$ and let $h $ be an integer.  If $\theta_i \notin \mathbb{Q}$ for all $i$, then
$$ \{ (\gamma_1 (p-h)^{\theta_1}, \gamma_2 (p-h)^{\theta_2}, \cdots , \gamma_m (p-h)^{\theta_m}) \}$$
is u.d. $\bmod 1$ in $\mathbb{T}^m$, where $p$ describes the increasing sequence of prime numbers belonging to the congruence class $t + q \mathbb{N}$. 
\end{Corollary}

The proof of Corollary \ref{ud along} hinges on the following classical identity (see p.34 in \cite{Mo}). \begin{Lemma} 
\label{mo}
For any $q \in \mathbb{N}$ and $b \in \mathbb{N}$ with $1 \leq b \leq q$, one has
\begin{displaymath}
  \frac{1}{q} \sum_{j=1}^q e\left(\frac{(n-b)j}{q} \right)= 
     \begin{cases}
       1 &  n \equiv b \,\, (\bmod q)\\
       0 &  \text{} \textrm{otherwise}
            \end{cases}
\end{displaymath} 
\end{Lemma}

\begin{proof}[Proof of Corollary \ref{ud along}]
Let $A_N = \{ p \leq N : p \equiv t \mod q \} $. We need to show that for $(a_1, a_2, \cdots, a_m) \ne (0,0, \cdots, 0)$,$$\lim_{N \rightarrow \infty} \frac{1}{|A_N|} \sum_{p \in A_N} e \left( \sum_{i=1}^m a_i \gamma_i (p-h)^{\theta_i} \right) = 0. $$
The result follows from 

 \begin{align*}
(i) \sum\limits_{\substack{ p \leq N \\ p \equiv t \,\,  \bmod q}} e\left( \sum_{i=1}^m a_i \gamma_i (p-h)^{\theta_i} \right) 
&= \sum_{p \leq N} e \left( \sum_{i=1}^m a_i \gamma_i (p-h)^{\theta_i} \right) \frac{1}{q} \sum_{j=1}^q e \left( \frac{(p-t)j}{q} \right) \\
&= \frac{1}{q} \sum_{j=1}^q \sum_{p \leq N} e \left( \sum_{i=1}^m a_i \gamma_i (p-h)^{\theta_i} + \frac{j}{q}(p-h) - \frac{j}{q}(t-h) \right)
\end{align*}

and 
$(ii)$ $$\lim_{N \rightarrow \infty} \frac{|A_N|}{\pi(N)} = \lim_{N \rightarrow \infty} \frac{\left| \{ p \leq N : p \equiv t \,\, \bmod q \} \right|}{\pi(N)} = \frac{1}{\phi(q)},$$
where $\phi$ is Euler's totient function.
\end{proof}

We will utilize Corollary \ref{lemma p-1} in the proof of the following proposition, which will be used in the next sections.
\begin{Proposition} 
\label{ud lemma}
Let $g(x) = \sum_{j=1}^m \alpha_j [x^{\theta_j}] $, where $\theta_1, \theta_2, \dots , \theta_m$ are distinct positive real numbers and $\alpha_1, \alpha_2, \dots, \alpha_m$ are non-zero reals. Let $h$ be an integer.
\begin{enumerate}[(i)]
\item If $\theta_j \notin \mathbb{Z}$ for all $j$ and $\alpha_j \notin \mathbb{Z}$ for all $j$, then 
 \begin{equation}
 \label{eqn1}
 \lim_{N \rightarrow \infty} \frac{1}{N} \sum_{n=1}^N e(g(p_n -h)) = 0.
 \end{equation}
\item If one of $\alpha_j$ is irrational, then $(g(p - h))_{p \in \mathcal{P}}$ is u.d. $\bmod 1$.
\end{enumerate}
\end{Proposition}

\begin{proof}
Our argument is similar to that used in the proof of Lemma 5.12 in \cite{BK2}. We will prove the case $h=0$ with the help of Theorem \ref{ud}. The case of non-zero $h$ can be done similarly by invoking Corollary \ref{lemma p-1} instead of Theorem \ref{ud} and is omitted.

($i$) Without loss of generality, we can assume that there exists $l$ such that 
$\alpha_1 , \dots , \alpha_l \notin  \mathbb{Q} \,\,\, \textrm{and} \,\,\, \alpha_{l+1}, \dots , \alpha_m \in \mathbb{Q}$. Furthermore we also assume that $\alpha_{l+1} , \dots , \alpha_m$ have a common denominator $q$, thus denote $\alpha_j = \frac{c_j}{q}$ for $l+1 \leq j \leq m$.

 We have  $$e(g(p_n)) = e{ ( [p_n^{\theta_1}] \alpha_1 + \cdots + [p_n^{\theta_m}] \alpha_m )}  = \prod_{j=1}^l f_j (p_n^{\theta_j} \alpha_j, p_n^{\theta_j}) \prod_{j=l+1}^m g_j([p_n^{\theta_j}]), $$
where $f_j(x,y) =  e(x - \{y\} \alpha_j) $ $(1 \leq j \leq l)$, and $g_j(z) = e ( c_j \frac{z}{q})$ $(l+1 \leq j \leq m)$.

Note that $f_j(x,y)$ are Riemann-integrable on $\mathbb{T}^2$ and $g_j(z)$ are continuous functions on $\mathbb{Z}_q = \mathbb{Z} / q \mathbb{Z}$, hence the function $\prod\limits_{j=1}^l f_j \prod\limits_{j=l+1}^m g_j $ is Riemann-integrable on $\mathbb{T}^{2l} \times \mathbb{Z}_q^{m-l}$. 

It follows from Theorem \ref{ud} and the classical Weyl criterion that, for any $u \in \mathbb{N}$,
 $$(p_n^{\theta_1} \alpha_1, p_n^{\theta_1}, \cdots , p_n^{\theta_l} \alpha_l, p_n^{\theta_l} , \frac{p_n^{\theta_{l+1}}}{u} , \cdots , \frac{p_n^{\theta_m}}{u}  )$$ is u.d. in $\mathbb{T}^{2l} \times \mathbb{T}^{m-l}$.
Since $[x]  \equiv a$ ($\bmod q$) is equivalent to $\frac{a}{q} \leq \{\frac{x}{q}\} < \frac{a+1}{q}$, we have 
$$(p_n^{\theta_1} \alpha_1, p_n^{\theta_1}, \cdots , p_n^{\theta_l} \alpha_l, p_n^{\theta_l} , [p_n^{\theta_{l+1}}] , \cdots , [p_n^{\theta_m}]  )$$
 is u.d. in $\mathbb{T}^{2l} \times \mathbb{Z}_q^{m-l}$.  Hence, (\ref{eqn1}) follows.

($ii$) By rearranging $\theta_i$, we can write
$$ g(x) = \sum_{i=1}^s a_i x^{\gamma_i} + \sum_{j=1}^t b_j [ x^{\delta_j}],  $$
where $\gamma_i \in \mathbb{N}$ and $\delta_j \in \mathbb{R}^+  \backslash \mathbb{N}$.

Then, for any non-zero integer $r$, we need to show 
$$\lim_{N \rightarrow \infty} \frac{1}{N} \sum_{n=1}^N e(r g(p_n)) = 0.$$

Without loss of generality, we assume that $b_1, \dots b_l$ are irrational and $b_{l+1}, \dots , b_{t} $ are rational. We also assume that $b_j = \frac{c_j}{q}$ $(j=l+1, \dots , t)$. Let $P(x) = \sum_{i=1}^s a_i x^{\gamma_i}$. Now consider the following two cases. 

Case I. Suppose that some $a_i$ is irrational. 
Note that
\begin{eqnarray*}
e(r g(p_n)) &=& e\left( r P(p_n) \right) \prod_{j=1}^l e \left( r b_j (p_n)^{\delta_j} - r b_j \{(p_n)^{\delta_j}\}  \right) \prod_{j=l+1}^t e \left( r b_j [(p_n)^{\delta_j}] \right) \\
&=& f_0(P(p_n)) \prod_{j=1}^l f_j(b_j(p_n)^{\delta_j}, (p_n)^{\delta_j}) \prod_{j=l+1}^t g_j([(p_n)^{\delta_j} ]),
\end{eqnarray*}
where $f_0(x) = e(r x)$, $f_j(x,y) = e(r (x-b_j\{y\}))$ $(1 \leq j \leq l)$, and $g_j(x) = e(r c_j \frac{x}{q}  )$ $(l+1 \leq j \leq t )$.
Using the above argument with Theorem \ref{ud} 
$$ \left( P(p_n) , b_1 (p_n)^{\delta_1}, (p_n)^{\delta_1}, \cdots , b_l (p_n)^{\delta_l}, (p_n)^{\delta_l}, [(p_n)^{\delta_{l+1}}] , \cdots ,  [(p_n)^{\delta_{t}}]  \right)$$
is uniformly distributed on $\mathbb{T}^{2l+1} \times \mathbb{Z}_q^{t-l}$.
Hence, $(g(p))_{p \in \mathcal{P}}$ is uniformly distributed $\bmod 1$.

Case II. Suppose that all $a_i$ are rational. Note that $b_1$ is irrational. Using the same method in Case I, the result follows from that
\begin{equation*}
\left( P(p_n) + b_1 (p_n)^{\delta_1}, (p_n)^{\delta_1}, b_2 (p_n)^{\delta_2}, (p_n)^{\delta_2}, \cdots , b_l (p_n)^{\delta_l}, (p_n)^{\delta_l}, [(p_n)^{\delta_{l+1}}], \cdots ,  [(p_n)^{\delta_{t}}]  \right)
\end{equation*}
is uniformly distributed on $\mathbb{T}^{2l} \times \mathbb{Z}_q^{t-l}$.
\end{proof}

\section{Recurrence along non-integer prime powers}
\label{sec : app}
 
 
 In this section we will prove the following ergodic theorem along the prime powers and derive some corollaries pertaining to sets of recurrence and sets of differences of positive upper Banach density in $\mathbb{Z}^k$.

\begin{Theorem}
\label{ergodic}
Let $c_1, \dots , c_k$ be distinct positive real numbers such that $c_i \notin \mathbb{N}$ for $i=1,2, \dots, k$. Let $U_1, \dots , U_k$ be commuting unitary operators on a Hilbert space $\mathcal{H}$.  Then,
$$\lim_{N \rightarrow \infty} \frac{1}{N}  \sum_{n=1}^N U_1^{[p_n^{c_1}]} \cdots U_k^{[p_n^{c_k}]} f = f^*,$$
where $p_n$ denotes the $n$-th prime and $f^*$ is the projection of $f$ on $\mathcal{H}_{inv} (:= \{ f \in \mathcal{H} :  U_i f = f \,\, \textrm{for all} \,\, i\})$.
 \end{Theorem}

For the proof of this theorem, we will need the following Hilbert space splitting theorem:
\begin{Theorem}[cf. \cite{B4}] 
\label{Hilbert space}
Let $U_1, U_2, \dots, U_k$ be commuting unitary operators on a Hilbert space $\mathcal{H}$. Then we can split $\mathcal{H}$ in the following ways.
\begin{enumerate}[(i)]
\item
$\mathcal{H} = \mathcal{H}_{inv} \oplus \mathcal{H}_{erg},$
 where
$$\mathcal{H}_{inv} = \{ f \in \mathcal{H} :  U_i f = f \,\, \textrm{for all} \,\, i\} , $$
 and
$$\mathcal{H}_{erg} = \{ f \in \mathcal{H} : \lim_{N_1, \cdots, N_k \rightarrow \infty} \left|\left|\frac{1}{N_1 \cdots N_k} \sum_{n_1=0}^{N_1-1} \cdots  \sum_{n_k=0}^{N_k-1}U_1^{n_1} \cdots U_k^{n_k} f \right|\right| = 0 \}.$$
 
 \item $\mathcal{H} = \mathcal{H}_{rat} \oplus \mathcal{H}_{tot},$
 where
$$\mathcal{H}_{rat} = \overline{ \{ f \in \mathcal{H} :  \textrm{there exists non-zero} \, k\textrm{-tuple} \,\, (m_1, m_2, \dots, m_k) \in \mathbb{Z}^k, \, U_i^{m_i} f = f \,\, \textrm{for all} \,\, i  \} }, $$
 and
\begin{eqnarray*}
\mathcal{H}_{tot} = \{ f \in \mathcal{H} &:&  \,\, \textrm{for any non-zero} \,\, (m_1, m_2,  \dots, m_k) \\
&&\lim_{N_1, \cdots, N_k \rightarrow \infty} \left|\left|\frac{1}{N_1 \cdots N_k} \sum_{n_1=0}^{N_1-1} \cdots  \sum_{n_k=0}^{N_k-1}U_1^{m_1 n_1} \cdots U_k^{m_k n_k} f \right|\right| = 0 \}.
\end{eqnarray*}
 \end{enumerate}
\end{Theorem}

We will also need the following version of the classical Bochner-Herglotz theorem.
\begin{Theorem}
\label{BH}
Let $U_1, \cdots, U_k$ be commuting unitary operators on a Hilbert space $\mathcal{H}$ and $f \in \mathcal{H}$. Then there is a measure $\nu_f$ on $\mathbb{T}^k$ such that 
$$ <U_1^{n_1} U_2^{n_2} \cdots U_k^{n_k} f , f >  \, \, = \int_{\mathbb{T}^k} e^{2 \pi i (n_1 \gamma_1 + \cdots + n_k \gamma_k)} \, d \nu_f (\gamma_1, \cdots , \gamma_k), $$
for any $(n_1, n_2, \cdots , n_k) \in \mathbb{Z}^k$.
\end{Theorem}

\begin{proof}[Proof of Theorem \ref{ergodic}]
Consider Hilbert space splitting $\mathcal{H} = \mathcal{H}_{inv} \oplus \mathcal{H}_{erg}$. For $f \in \mathcal{H}_{inv}$, $U_1^{[p_n^{c_1}]} \cdots U_k^{[p_n^{c_k}]} f = f$. So let us assume that $f \in \mathcal{H}_{erg}$ and show that 
\begin{equation}
\label{averaging sum}
 \lim_{N \rightarrow \infty} \frac{1}{N} \sum_{n=1}^N U_1^{[p_n^{c_1}]} \cdots U_k^{[p_n^{c_k}]} f   = 0 . 
 \end{equation}

This will follow from Proposition \ref{ud lemma} and Theorem \ref{BH}. We have
\begin{eqnarray*}
 \left|\left| \frac{1}{N} \sum_{n=1}^N U_1^{[p_n^{c_1}]} \cdots U_k^{[p_n^{c_k}]}  f \right|\right|_2^2 
 &=& \frac{1}{N^2} \sum_{m,n=1}^N \langle U_1^{[p_m^{c_1}]} \cdots U_k^{[p_m^{c_k}]}  f, U_1^{[p_n^{c_1}]} \cdots U_k^{[p_n^{c_k}]} f \rangle \\
 &=&\frac{1}{N^2} \sum_{m,n=1}^N \langle U_1^{[p_m^{c_1}]  - [p_n^{c_1}]}  \cdots U_k^{[p_m^{c_k}] -  [p_n^{c_k}]} f, f \rangle \\
 &=&\frac{1}{N^2} \sum_{m,n=1}^N \int e(([p_m^{c_1}] - [p_n^{c_1}] , \cdots ,  [p_m^{c_k}] -  [p_n^{c_k}]) \cdot \bold{\gamma}) \, d \nu_f (\bold{\gamma})\\
&=& \int \left| \frac{1}{N} \sum_{n=1}^N e ( ([p_n^{c_1}],  \cdots , [p_n^{c_k}]) \cdot \bold{\gamma}) \right|^2 \, d \nu_f(\bold{\gamma})
\end{eqnarray*}
Since $f \in H_{erg}$, we have $\nu_f (\{(0, \dots, 0)\})  = 0$, so that, for our $f$, \eqref{averaging sum} follows. 
\end{proof}

\begin{Corollary}
\label{ergodic sequence}
Let $c_1, c_2, \dots , c_k$ be positive, non-integers. Let $T_1, T_2, \dots , T_k$ be commuting, invertible measure preserving transformations on a probability space $(X, \mathcal{B}, \mu)$. 
Then, for any $A \in \mathcal{B}$ with $\mu(A) > 0$, one has 
$$\lim_{N \rightarrow \infty} \frac{1}{N} \sum_{n=1}^{ N} \mu(A \cap T_1^{-[p_{n}^{c_1}]} \cdots T_k^{-[p_n^{c_k}]} A) \geq \mu^2(A).$$
\end{Corollary}
\begin{proof}
Without loss of generality we can assume that $c_1, c_2, \dots , c_k$ are distinct by regrouping and collapsing some of the $T_i$. (For example, if $c_1= c_2$, $T_1^{-[p_n^{c_1}]} T_2^{-[p_n^{c_2}]} = (T_1 T_2)^{- [p_n^{c_1}]}$).
Let $f= 1_A$. A measure preserving transformation $T_i$ can be considered as a unitary operator  $T_i f = f \circ T_i $. Denote by $P$ the projection on $ \mathcal{H}_{inv}$ for $T_1 , \dots , T_k$. Then we have
\begin{eqnarray*}
 &&\lim_{N \rightarrow \infty} \frac{1}{N} \sum_{n=1}^N \mu(A \cap T_1^{-[p_{n}^{c_1}]} \cdots T_k^{-[p_n^{c_k}]}A) =  \lim_{N \rightarrow \infty} \frac{1}{N} \sum_{n=1}^N \int f \, T_1^{[p_{n}^{c_1}]} \cdots T_k^{[p_n^{c_k}]} f \, d \mu \\
&&= \int f Pf \, d\mu = \langle f, Pf \rangle = \langle f, P^2 f \rangle = \langle Pf, Pf \rangle 
\geq \left( \int Pf \, d \mu \right)^2 = \mu^2(A).
\end{eqnarray*} 
\end{proof}

Recall that the upper Banach density of a set $E \subset \mathbb{Z}^k$ is defined to be
$$d^*(E) = \sup_{\{\Pi_n\}_{n \in \mathbb{N}}} \limsup_{n \rightarrow \infty} \frac{|E \cap \Pi_n|}{| \Pi_n|},$$
where the supremum is taken over all sequences of parallelepipeds
$$ \Pi_n = [a_n^{(1)}, b_n^{(1)}] \times \cdots \times [a_n^{(k)}, b_n^{(k)} ] \subset \mathbb{Z}^k, \,\, n \in \mathbb{N},$$
with $b_n^{(i)} - a_n^{(i)} \rightarrow \infty$, $1 \leq i \leq k$.

By the $\mathbb{Z}^k$-version of Furstenberg's correspondence principle (see, for example, Proposition 7.2 in \cite{BMc}), 
given $E \subset \mathbb{Z}^k$ with $d^*(E)>0$, there is a probability space $(X, \mathcal{B} , \mu)$, commuting invertible measure preserving transformations $T_1, T_2, \dots , T_k$ of $X$ and $A \in \mathcal{B} $ with $d^*(E) = \mu(A)$ such that for any $\bold{n_1}, \bold{n_2}, \dots , \bold{n_m} \in \mathbb{Z}^k $ one has
$$d^*(E \cap (E- \bold{n_1}) \cap (E- \bold{n_2}) \cap \cdots \cap (E - \bold{n_m})) \geq \mu(A \cap T^{- \bold{n_1}}A \cap \cdots \cap T^{-\bold{n_m}} A),$$
where for $\bold{n} = (n_1, \dots , n_k)$, $T^{\bold{n}} = T_1^{n_1} \cdots T_k^{n_k}.$

We see now that Corollary \ref{ergodic sequence} together with Furstenberg's correspondence principle implies the following result.

\begin{Corollary}
\label{prop}
Let $c_1, \cdots ,c_k$ be positive non-integers.
If $E \subset \mathbb{Z}^k$ with ${d^*}(E) > 0$, then there exists a prime $p$ such that $([p^{c_1}], \cdots , [p^{c_k}] ) \in E - E$. 
Moreover,
$$\liminf_{N \rightarrow \infty} \frac{ | \{ p \leq N:   ([p^{c_1}], \cdots , [p^{c_k}] ) \in E - E\} | }{ \pi(N) } \geq {d^*}(E)^2.$$
\end{Corollary}

\begin{proof}
By a special case of Furstenberg's correspondence principle, given $E \subset \mathbb{Z}^k$ with $d^*(E)>0$, there exist a probability space $(X, \mathcal{B} , \mu)$, commuting invertible measure preserving transformations $T_1, \dots , T_k$ of $X$ and $A \in \mathcal{B} $ with $d^*(E) = \mu(A)$ such that for any $l_1, l_2, \dots , l_k \in \mathbb{Z} $ one has
$$d^*(E \cap (E- (l_1, l_2, \cdots, l_k)) \geq \mu(A \cap T_1^{-l_1} T_2^{-l_2} \cdots T_k^{-l_k}A).$$
Note that 
\begin{align*}
\left| \{ p \leq N:   ([p^{c_1}], \cdots , [p^{c_k}] ) \in E - E\} \right| &\geq \left| \{ p \leq N:  d^*(E \cap E - ([p^{c_1}], \cdots , [p^{c_k}] ) >0 \} \right| \\
&\geq \sum_{p \leq N} d^*(E \cap E - ([p^{c_1}], \cdots , [p^{c_k}] ) )\\
&\geq \sum_{p \leq N} \mu(A \cap T_1^{-[p^{c_1}]} T_2^{-[p^{c_2}]} \cdots T_k^{-[p^{c_k}]} A).
 \end{align*}
Hence, by Corollary \ref{ergodic sequence},
\begin{align*}
&\liminf_{N \rightarrow \infty} \frac{ \left| \{ p \leq N:   ([p^{c_1}], \cdots
     , [p^{c_k}] ) \in E - E\} \right| }{\pi(N)}\\
&\qquad\geq \lim_{N \rightarrow
   \infty} \frac{1}{\pi(N)} \sum_{p \leq N} \mu(A \cap T_1^{-[p^{c_1}]}
 T_2^{-[p^{c_2}]} \cdots T_k^{-[p^{c_k}]} A) \\ 
&\qquad\geq \mu(A)^2 = d^*(E)^2.
 \end{align*}
\end{proof}

\begin{Remark}
It is not hard to see that Theorem \ref{ergodic}, Corollary \ref{ergodic sequence} and Corollary \ref{prop} remain true if one replaces in the formulations $([p^{c_1}], \cdots , [p^{c_k}] )$ by $([(p-h)^{c_1}], \cdots , [(p-h)^{c_k}] ) $ for any integer $h$. We will utilize this remark for $h= \pm 1$ in the next section. 
\end{Remark}

\section{Application to Nice $FC^+$ sets}


\begin{Definition}
A sequence $(\bold{d}_n)_{n \in \mathbb{N}}$ in $\mathbb{Z}^k$ is called {\it{ergodic}} if the following mean ergodic theorem is valid: for any ergodic measure preserving $\mathbb{Z}^k$-action $T=(T^{\bold{m}})_{(\bold{m} \in \mathbb{Z}^k)}$ on a probability space $(X, \mathcal{B}, \mu)$, 
$$\lim_{N \rightarrow \infty} \frac{1}{N} \sum_{n=1}^N f \circ T^{\bold{d}_n} = \int f \, d \mu  \,\,\, \textrm{for any} \,\, f \in L^2(\mu).$$ 
\end{Definition}
Recall that a subset $D$ of $\mathbb{Z}^k$ is {\it{a set of recurrence}} if given any measure preserving $\mathbb{Z}^k$-action $T=(T^{\bold{m}})_{(\bold{m} \in \mathbb{Z}^k)}$ on a probability space $(X, \mathcal{B}, \mu)$ and any set $A \in \mathcal{B}$ with $\mu(A) > 0$, there exists $\bold{d} \in D$ $(\bold{d} \ne 0)$ such that 
$$\mu(A \cap T^{-\bold{d}} A) > 0. $$

\begin{Definition} 
Let $D$ be a subset of $\mathbb{Z}^k$.
We will write $D= \{\bold{d}_n : n \in \mathbb{N} \}$ with the convention that $\bold{d}_n$ are pairwise distinct and the sequence $(|\bold{d}_n|)$ is non-decreasing. (Here $|\bold{d}| = \sup_{1 \leq i \leq k} |d_i|$ for $\bold{d} = (d_1, d_2, \cdots, d_k)$.)
\begin{enumerate}
\item (cf.\cite{B2})
A set $D$ is {\it{a set of nice recurrence}} if given any measure preserving $\mathbb{Z}^k$-action $T=(T^{\bold{m}})_{(\bold{m} \in \mathbb{Z}^k)}$ on a probability space $(X, \mathcal{B}, \mu)$, any set $A \in \mathcal{B}$ with $\mu(A) > 0$ and any $\epsilon >0$, we have 
$$\mu(A \cap T^{-\bold{d}} A) \geq \mu^2(A) - \epsilon $$
for infinitely many $\bold{d} \in D$.
\item  (cf.\cite{BK1} and \cite{BL}) A set $D$ is {\it{an averaging set of recurrence}} if given any measure preserving $\mathbb{Z}^k$-action $T=(T^{\bold{m}})_{(\bold{m} \in \mathbb{Z}^k)}$ on a probability space $(X, \mathcal{B}, \mu)$ and any set $A \in \mathcal{B}$ with $\mu(A) > 0$  we have 
$$\limsup_{N \rightarrow \infty } \frac{1}{N} \sum_{n=1}^N \mu(A \cap T^{-\bold{d}_n} A) > 0 .$$
\end{enumerate}
\end{Definition} 

\begin{Definition}[cf.\cite{BL}, Definition 1.2.1] 
A subset $D$ of $\mathbb{Z}^k \backslash \{0\}$ is {\it{a van der Corput set}} ({\it{vdC set}}) if for any family $(u_{\bold{n}})_{\bold{n} \in \mathbb{Z}^k}$ of complex numbers of modulus $1$ such that 
$$\forall \bold{d} \in D, \,\,  \lim_{N_1, \cdots, N_k \rightarrow \infty} \frac{1}{N_1 \cdots N_k} \sum_{\bold{n} \in \prod_{i=1}^k [0, N_i)} u_{\bold{n}+\bold{d}} \overline{u_{\bold{n}}} = 0$$
we have $$\lim_{N_1, \cdots, N_k \rightarrow \infty} \frac{1}{N_1 \cdots N_k } \sum_{\bold{n} \in \prod_{i=1}^k [0, N_i)} u_{\bold{n}} = 0.$$
 \end{Definition}

\begin{Definition}[\cite{BL}]
An infinite set $D$ of $\mathbb{Z}^k$ is {\it{a nice $FC^+$ set}} if for any positive finite measure $\sigma$ on $\mathbb{T}^k$,
$$\sigma( \{ (0,0, \cdots, 0) \} ) \leq \limsup_{|\bold{d}| \rightarrow \infty, \bold{d} \in D} |\hat{\sigma}(\bold{d})|. $$
 \end{Definition} 
 
\begin{Remark}  The following results are obtained in \cite{BL} for sets in $\mathbb{Z}$ and can be generalized to $\mathbb{Z}^k$ .
\begin{enumerate} 
\item An ergodic sequence in $\mathbb{Z}^k$ is an averaging set of recurrence and a set of nice recurrence.
 This can be obtained by using the same argument as in the proof of Corollary \ref{ergodic sequence}.
\item A nice $FC^+$ set in $\mathbb{Z}^k$ is a set of nice recurrence. The proof for $\mathbb{Z}$  was given in \cite{BL}. Here we add the proof for reader's convenience. Let $T = (T^{\bold{n}})_{\bold{n} \in \mathbb{Z}^k}$ be a measure preserving $\mathbb{Z}^k$-action on a probability space $(X, \mathcal{B}, \mu)$ and $A \in \mathcal{B}$. Then there exists a positive measure $\sigma$ on $\mathbb{T}^k$ such that $\hat{\sigma}(\bold{n}) = \mu (A \cap T^{- \bold{n}}A)$ and $\sigma(\{ (0,0, \cdots, 0)\}) \geq \mu(A)^2$. Then the result follows. 
\item A nice $FC^+$ set is a vdC set. This is an immediate consequence of the following spectral characterization of vdC sets (see Theorem 1.8 in \cite{BL}):
  A set $D \subset \mathbb{Z}^k$ is vdC  if and only if any positive measure $\sigma$ on $\mathbb{T}^k$ with $\hat{\sigma}(\bold{d}) = 0$ for all $\bold{d} \in D$ satisfies $\sigma(\{(0,0, \cdots,0 \}) = 0.$
\end{enumerate}

\end{Remark}

 Next we also obtain the following result, which can be viewed as an extension of S\'{a}rk{\H o}zy's Theorem. (See \cite{Sa1} and \cite{Sa3}.)
\begin{Theorem}
\label{sarkozy type}
If $\alpha_i$ are positive integers and $\beta_i$ are positive and non-integers, then
$$ D_1 = \{ \left(  (p-1)^{\alpha_1}, \cdots , (p-1)^{\alpha_k}, [(p-1)^{\beta_1}], \cdots , [(p-1)^{\beta_l}]  \right) | \, p \in\mathcal{P} \},$$
and
$$ D_2 = \{ \left(  (p+1)^{\alpha_1}, \cdots , (p+1)^{\alpha_k}, [(p+1)^{\beta_1}], \cdots , [(p+1)^{\beta_l}]  \right) | \, p \in \mathcal{P} \}$$
are nice $FC^+$ sets in $\mathbb{Z}^{k+l}$, and so they are vdC sets and also sets of nice recurrence.
\end{Theorem}
\begin{Remark}
Recall that a set $D$ of positive integers is a van der Corput set (or vdC set) if given a real sequence $(x_n)_{n \in \mathbb{N}}$, equidistribution $\bmod 1$ of $(x_{n+d} -x_n)_{n \in \mathbb{N}}$ for all $d \in D$ implies the equidistribution of  $(x_n)_{n \in \mathbb{N}}$.
Let $\mathcal{P}$ be the set of all prime numbers. It is shown in \cite{KM} that $\mathcal{P} -h$ is a vdC set if and only if $h = \pm 1$. 
Since a nice $FC^+$ set is a vdC set (see section 3.5 in \cite{BL}), we cannot replace $\pm 1$ by any other integer $h$ on Theorem \ref{sarkozy type}.
\end{Remark}

The following lemma, which tells us how to recognize a nice $FC^+$ set, will be utilized in the proof of Theorem \ref{sarkozy type}.
\begin{Lemma}[cf. \cite{BL} Proposition 2.11]
\label{vdc}
Let $D \subset \mathbb{Z}^k$.  For each $q \in \mathbb{N}$, define 
$$D_q := \{ \bold{d} = (d_1, d_2, \dots, d_k) \in D : q! \,\, \textrm{divides} \,\, d_i \,\, \textrm{for} \,\, 1 \leq i \leq k \}.$$
Suppose that, for every $q$, there exists a sequence $(\bold{d}^{q,n})_{n \in \mathbb{N}}$ in $D_q$ such that 

($i$) $|\bold{d}^{q,n}|$ is non-decreasing and 
($ii$) for any $\bold{x}=(x_1, \cdots , x_k) \in \mathbb{R}^k$, if one of $x_i$ is irrational, the sequence $(\bold{x} \cdot \bold{d}^{q,n})_{n \in \mathbb{N}}$ is uniformly distributed $\bmod 1$.

Then $D$ is a nice $FC^+$ set.
\end{Lemma}

\begin{proof}
For simplicity of notation we will confine ourselves to the case $k=1$. In this case we write $d^{q,n}$ for $\bold{d}^{q,n}$.

We need to show that,  for any positive finite measure $\sigma$ on $\mathbb{T}$,
$$\sigma(\{0\}) \leq \limsup_{d \in D, |d| \rightarrow \infty } |\hat{\sigma} (d)| . $$

Given $q$, define $f_N(x) = \frac{1}{N} \sum_{n=1}^N e(d^{q,n} x)$.
Let $A_q = \{ \frac{a}{q!} : 0 \leq a \leq q!-1, a \in \mathbb{N} \} \subset \mathbb{T}$ and let $B_q = \{ r \in \mathbb{T} \cap \mathbb{Q} : r \notin A_q \}$. 
Then $\lim\limits_{N \rightarrow \infty} f_N(x) = 0$ if $x$ is irrational and $\lim\limits_{N \rightarrow \infty} f_N(x) = 1$ if $x \in A_q$. Since $B_q$ is countable, we can choose a sequence $N_j$ such that $\lim\limits_{N_j \rightarrow \infty} f_{N_j}(x)$ exists for every $x \in B_q$, thus for every $x \in \mathbb{T}$.  
Let
$f(x) := \lim\limits_{N_j \rightarrow \infty} f_{N_j}(x)$. Note that $0 \leq |f(x)| \leq 1$ for all $x$.

By the dominated convergence theorem,
\begin{equation}
\label{fc eqn1}
\int_{\mathbb{T}} f(x) \, d \sigma = \lim_{N_j \rightarrow \infty} \frac{1}{N_j} \sum_{n=1}^{N_j} \int e(d^{q,n}x ) \, d \sigma = \lim_{N_j \rightarrow \infty} \frac{1}{N_j} \sum_{n=1}^{N_j} \hat{\sigma}(d^{q,n}).
\end{equation}

Since $f(x) = 0$ for $x \in \mathbb{T} \backslash \mathbb{Q}$,  
\begin{align}
\label{fc eqn2}
\left| \int_{\mathbb{T}} f(x) \, d \sigma \right| &= \left| \int_{A_q} f(x) \, d \sigma +  \int_{B_q} f(x) \, d \sigma +
 \int_{\mathbb{T} \backslash \mathbb{Q}} f(x) \, d \sigma\right| \nonumber \\
  &=  \left| \int_{A_q} f(x) \, d \sigma +  \int_{B_q} f(x) \, d \sigma \right| \geq \int_{A_q} f(x) \, d \sigma -  \int_{B_q} | f(x)| \, d \sigma \nonumber \\
 &\geq \sigma(A_q) - \sigma(B_q). 
   \end{align}
Also we have 
\begin{align}
\label{fc eqn3}
\limsup\limits_{d \in D, |d| \rightarrow \infty} |\hat{\sigma} (d)| &\geq \limsup\limits_{n \rightarrow \infty} |\hat{\sigma} (d^{q,n})| \nonumber \\
&\geq  \limsup\limits_{N_j \rightarrow \infty} \frac{1}{N_j} \sum_{n=1}^{N_j} | \hat{\sigma} (d^{q,n}) | 
\geq \left| \lim\limits_{N_j \rightarrow \infty} \frac{1}{N_j} \sum_{n=1}^{N_j} \hat{\sigma} (d^{q,n}) \right|. 
\end{align}

From equations (\ref{fc eqn1}), (\ref{fc eqn2}) and (\ref{fc eqn3}),
$$\sigma(A_q) - \sigma(B_q) \leq \limsup\limits_{d \in D, |d| \rightarrow \infty} |\hat{\sigma} (d)|.$$
By the continuity of the measure, $\lim\limits_{q \rightarrow \infty} \sigma(A_q) = \sigma (\mathbb{T} \cap \mathbb{Q})$ and $\lim\limits_{q \rightarrow \infty} \sigma(B_q) = 0$. So,
$$\sigma(\{0\}) \leq \sigma(\mathbb{T} \cap \mathbb{Q}) \leq \limsup_{d \in D, |d| \rightarrow \infty} |\hat{\sigma} (d)|.$$
  \end{proof}

\begin{Proposition}
\label{positive density}
Let $D_1$ and $D_2$ be as in Theorem \ref{sarkozy type} and $D_i^{(r)} = D_i \bigcap (\bigoplus_{j=1}^{k+l}  r \mathbb{Z})$. 
Then $D_i^{(r)}$ has positive relative density in $D_i$ for $i=1,2$. 
\end{Proposition}

\begin{proof}
Let us prove this for $D_1$. Without loss of generality we can assume that all $\beta_i$ are distinct. Note that $\mathcal{P} = \bigcup\limits_{(t,r) =1}( (t+ r\mathbb{Z}) \bigcap \mathcal{P})$ and the relative density of $(t+ r\mathbb{Z}) \bigcap \mathcal{P}$ in $\mathcal{P}$ is $\frac{1}{\phi(r)}$.
Now, if $p \in (t + r \mathbb{Z}) \bigcap \mathcal{P}$, the pair of conditions 
$$ r| (t-1)^{\alpha_i} \, (1 \leq i \leq k) \,\, \text{and} \,\, 0 \leq \left\{ \frac{(p -1)^{\beta_i}}{r} \right\} < \frac{1}{r} \, (1 \leq i \leq l)$$
is equivalent to $\left(  (p-1)^{\alpha_1}, \cdots , (p-1)^{\alpha_k}, [(p-1)^{\beta_1}], \cdots , [(p-1)^{\beta_l}]  \right) \in D_1.$ 
The result follows from that
$ \left( \frac{(p-1)^{\beta_1}}{r}, \cdots ,  \frac{(p-1)^{\beta_l}}{r} \right)$
is uniformly distributed $\bmod 1$ in $\mathbb{T}^l$ along the increasing sequence of primes $p \in t + r \mathbb{Z}$.
The proof for $D_2$ is completely analogous.
\end{proof}

\begin{proof}[Proof of Theorem \ref{sarkozy type}]
Let us prove that $D_1$ is a nice $FC^+$ set. 

Denote $D_1 = (\bold{d}_n)_{n \in \mathbb{N}}$, where $$\bold{d}_n = \left(  (p_n - 1)^{\alpha_1}, \cdots , (p_n - 1)^{\alpha_k}, [(p_n - 1)^{\beta_1}], \cdots , [(p_n - 1)^{\beta_l}]  \right).$$

Enumerate the elements of  $D_1^{(q!)}$ by $ (\bold{d}^{q,n})_{n \in \mathbb{N}}$, where $|\bold{d}^{q,n}|$ is non-decreasing. From Lemma \ref{vdc}, it is sufficient to show that for any $\bold{x} = (x_1, x_2, \cdots, x_{k+l})$, if one of $x_i$ is irrational, $(\bold{d}^{q,n} \cdot \bold{x})_{n \in \mathbb{N}} $ is u.d. $\bmod 1$.

 For any non-zero integer $h$, by Lemma \ref{mo},
\begin{align*}
 &\frac{1}{ | \{ n \leq N : \bold{d}_n \in D_1^{(q!)} \} | } \sum_{n \leq N, \bold{d}_n \in D_1^{(q!)}  } e (h (\bold{d}^{q,n} \cdot \bold{x}))  \\
&= \frac{1}{ | \{ n \leq N : \bold{d}_n \in D_1^{(q!)} \} | } \sum_{n \leq N  } e (h (\bold{d}_{n} \cdot \bold{x})) \frac{1}{(q!)^{k+l}} \sum_{j_1 = 1}^{q!} \cdots \sum_{j_{k+l} = 1}^{q!} e \left(\bold{d}_n \cdot \left(\frac{j_1}{q!}, \cdots , \frac{j_{k+l}}{q!} \right) \right) \\
&= \frac{N}{ | \{ n \leq N : \bold{d}_n \in D_1^{(q!)} \} | } \frac{1}{(q!)^{k+l}} \sum_{j_1 = 1}^{q!} \cdots \sum_{j_{k+l} = 1}^{q!} \frac{1}{N} \sum_{n \leq N}e \left( \bold{d}_n \cdot \left(h \bold{x} +\left(\frac{j_1}{q!}, \cdots , \frac{j_{k+l}}{q!} \right) \right) \right).
\end{align*}
Then the result follows from Proposition \ref{ud lemma} and Proposition \ref{positive density}.
The proof for $D_2$ is completely analogous.
 \end{proof}

\begin{Corollary}
\label{cor nice}
Let $D_1$ and $D_2$ be as in Theorem \ref{sarkozy type}.
If $E \subset \mathbb{Z}^{k+l}$ with ${d^*}(E) > 0$, then for any $\epsilon > 0$
\begin{equation*}
 R_i(E, \epsilon) := \{ \bold{d} \in D_i : d^*(E \cap E - \bold{d} ) \geq d^*(E)^2 - \epsilon \} 
 \end{equation*}
 is infinite for $i=1, 2$.
 \end{Corollary}

We will see in the next section that the sets $R_i(E, \epsilon)$ actually have positive lower relative density.
 
\section{Uniform distribution and sets of recurrence}

\begin{Theorem}
Let $D_1$ and $D_2$ be as in Theorem \ref{sarkozy type} and enumerate the elements of $D_1$ or $D_2$ as follows (where the sign $-$ corresponds to $D_1$ and sign $+$ corresponds to $D_2$):
 $$\bold{d}_n = \left(  (p_n \pm 1)^{\alpha_1}, \cdots , (p_n \pm 1)^{\alpha_k}, [(p_n\pm 1)^{\beta_1}], \cdots , [(p_n \pm 1)^{\beta_l}]  \right).$$
For each $r \in \mathbb{N}$, let $D_i^{(r)} = D_i \cap \bigoplus\limits_{j=1}^{k+l} r \mathbb{Z}$ and enumerate the elements of $D_i^{(r)}$ by $(\bold{d}_n^{(r) })$ such that  $|\bold{d}_{n}^{(r)}|$ is non-decreasing. 
Let $(T^{\bold{d}})_{\bold{d} \in \mathbb{Z}^{k+l}}$ be a measure preserving $\mathbb{Z}^{k+l}$-action on a probability space $(X, \mathcal{B}, \mu)$.
Then for $A \in \mathcal{B}$ with $\mu(A) > 0$ and $\epsilon > 0$, there exists $r \in \mathbb{N}$ such that
  \begin{equation}
  \label{eqn7}
 \lim_{N \rightarrow \infty} \frac{1}{N} \sum_{n=1}^N \mu(A \cap T^{-\bold{d}_n^{(r)}} A) \geq  \mu(A)^2 - \epsilon.
    \end{equation} 
Moreover,
\begin{enumerate}[(i)]
\item
\begin{equation}
\label{eqn8}
 \{ \bold{d} \in D_i : \mu(A \cap T^{-\bold{d}}A ) \geq \mu^2(A) - \epsilon \} 
 \end{equation}
 has positive lower relative density in $D_i$ for $i=1, 2$. Hence, $D_1$ and $D_2$ are sets of nice recurrence.

\item
  \begin{equation}
  \label{eqn9}
 \lim_{N \rightarrow \infty} \frac{1}{N} \sum_{n=1}^N \mu(A \cap T^{-\bold{d}_n} A) > 0.
  \end{equation} 
Thus $D_1$ and $D_2$ are averaging sets of recurrence.
 \end{enumerate}
 \end{Theorem}

\begin{proof} We will prove this result for $D_1$. (The proof for $D_2$ is similar.)
For $\mathbb{Z}^{k+l}$-action $T$, there are commuting measure preserving transformations $T_1, \cdots , T_{k+l}$ such that $T^{\bold{m}} = T_1^{m_1} \cdots T_{k+l}^{m_{k+l}}$ for $\bold{m} = (m_1, m_2 , \cdots , m_{k+l})$.

First we will show that 
\begin{equation}
\label{existence1}
 \lim\limits_{N \rightarrow \infty}\frac{1}{N} \sum\limits_{n=1}^N \mu(A \cap T^{-\bold{d}_n}A)
\end{equation} 
and 
\begin{equation}
\label{existence2}
\lim\limits_{N \rightarrow \infty} \frac{1}{N}
\sum_{n=1}^N \mu(A \cap T^{-\bold{d}_n^{(r)}} A)
\end{equation}
exist. 
 
 By Theorem \ref{BH}, there exists a measure $\nu$ on $\mathbb{T}^{k+l}$ such that
$$\mu(A \cap T^{-\bold{n}}A) = \int 1_A \,\, T^{\bold{n}} 1_A \, d \mu = \int_{\mathbb{T}^{k+l}} e(\bold{n} \cdot \bold{\gamma}) \, d \nu (\bold{\gamma}).$$
Thus, in order to prove that (\ref{existence1}) and (\ref{existence2}) exist, it is sufficient to show that for every
$\gamma$, 
$$\lim\limits_{N \rightarrow \infty}
\frac{1}{N} \sum\limits_{n=1}^N e(\bold{d}_n \cdot \bold{\gamma}) \,\,\,\, \textrm{and} \,\,\,\,
 \lim\limits_{N \rightarrow \infty} \frac{1}{N}
\sum\limits_{n=1}^N e(\bold{d}_n^{(r)} \cdot \bold{\gamma})$$ exist. 
Moreover, by Lemma \ref{mo}, denoting $A_N =\{ n \leq N: \bold{d}_n \in D_1^{(r)} \} $
\begin{align*}
&\lim_{N \rightarrow \infty} \frac{1}{N} \sum_{n=1}^N e (\bold{d}_n^{(r)} \cdot \bold{\gamma})\\
&= \lim_{N \rightarrow \infty} \frac{1}{|A_N|} \sum_{n=1}^N e\left(\bold{d}_n \cdot \bold{\gamma}\right) \,\, \left( \frac{1}{r} \sum_{j_1=1}^r e\left(\frac{(p_n-1)^{\alpha_1} j_1}{r}\right) \right) \cdots \left(\frac{1}{r} \sum_{j_{k+l}=1}^r e\left(\frac{[(p_n-1)^{\beta_l}] j_{k+l}}{r}\right) \right) \\
&= \lim_{N \rightarrow \infty} \frac{N}{|A_N|} \frac{1}{r^{k+l}} \sum_{j_1 = 1}^r \cdots \sum_{j_{k+l} = 1}^{r} \frac{1}{N}
\sum_{n=1}^N e \left(\bold{d}_n \cdot (\bold{\gamma} + \left(\frac{j_1}{r} +     \cdots + \frac{j_{k+l}}{r} \right) \right).
\end{align*}
Hence we only need to show that $\lim\limits_{N \rightarrow \infty}
\frac{1}{N} \sum\limits_{n=1}^N e(\bold{d}_n \cdot \bold{\gamma})$ exists for
every $\gamma$.  From Proposition \ref{ud lemma}, if $\bold{\gamma} \notin
\mathbb{Q}^{k+l}$, $\lim\limits_{N \rightarrow \infty} \frac{1}{N}
\sum\limits_{n=1}^N e(\bold{d}_n \cdot \bold{\gamma}) = 0.$ If $\bold{\gamma} =
(\gamma_1, \gamma_2, \cdots, \gamma_{k+l}) \in \mathbb{Q}^{k+l}$, then we can
find a common denominator $q \in \mathbb{N}$ for $\gamma_1, \dots,
\gamma_{k+l}$ such that $\gamma_i = \frac{a_i}{q}$ for each $i$. 
Then
\begin{align*} 
&\lim\limits_{N \rightarrow \infty} \frac{1}{N} \sum_{n=1}^N e(\bold{d}_n \cdot
\bold{\gamma})\\
&\quad= \lim_{N \rightarrow \infty} \frac{1}{\pi(N)} \sum\limits_{p \leq N} e\left(\sum_{i=1}^k (p-1)^{\alpha_i} \frac{a_i}{q} + \sum_{j=1}^l [(p-1)^{\beta_j}] \frac{a_{k+j}}{q}\right) \\
&\quad=\lim\limits_{N \rightarrow \infty}  \frac{1}{\pi(N)} \sum\limits_{\substack{ (t, q) =1 \\ 0 \leq t \leq q-1}} \sum\limits_{\substack{ p \equiv t \,\, \bmod q \\  p \leq N}} e\left(\sum_{i=1}^k (p-1)^{\alpha_i} \frac{a_i}{q} + \sum_{j=1}^l [(p-1)^{\beta_j}] \frac{a_{k+j}}{q}\right) \\
&\quad= \sum\limits_{\substack{ (t, q) =1 \\ 0 \leq t \leq q-1}} e\left(\sum_{i=1}^k (t-1)^{\alpha_i} \frac{a_i}{q}\right) \lim\limits_{N \rightarrow \infty} \frac{1}{\pi(N)} \sum\limits_{\substack{ p \equiv t \,\, \bmod q \\  p \leq N}} e\left(\sum_{j=1}^l [(p-1)^{\beta_j}] \frac{a_{k+j}}{q}\right). 
\end{align*}

We claim that 
$$\lim\limits_{N \rightarrow \infty}\frac{1}{\pi(N)} \sum\limits_{\substack { p
    \equiv t \, \, \bmod q \\  p \leq N}}  e\left(\sum_{j=1}^l
  [(p-1)^{\beta_j}] \frac{a_{k+j}}{q}\right)$$ exists. Without loss of
generality we assume that all $\beta_i$ are distinct. Then the claim is a consequence of following two facts: 
\begin{enumerate}[(i)]
\item $([(p-1)^{\beta_1}] , \cdots , [ (p-1)^{\beta_l} ] )$ is u.d. in $\mathbb{Z}_q^l$ along $p \in t + q \mathbb{Z}$ for $(t,q) = 1$, since $\left( \frac{(p-1)^{\beta_1}}{q}, \cdots , \frac{(p-1)^{\beta_l}}{q} \right)$
is u.d. $\bmod 1$ in $\mathbb{T}^l$ along $p \in t + q \mathbb{Z}$ from Corollary \ref{ud along}. 
\item $\{ p \in \mathcal{P} : p \equiv t \,\, \bmod q \}$ has a density $\frac{1}{\phi(q)}$ in $\mathcal{P}$ for $(t,q) = 1$.
\end{enumerate}

Now let us show (\ref{eqn7}). Applying Theorem \ref{Hilbert space} to (unitary operators induced by) $T_1, \cdots, T_{k+l}$ we have $1_A = f + g$, where $f \in \mathcal{H}_{rat} $ and $g \in \mathcal{H}_{tot}$.
Note that $\mathcal{H}_{rat} = \overline{\bigcup_{q=1}^{\infty} \mathcal{H}_q}$, where $\mathcal{H}_q = \{  f : T_i^{q!} f = f \,\, \textrm{for} \,\, i=1,2, \dots, k+l \}$.

 For $\epsilon >0$, there exists $\bold{a} = (a_1, \cdots , a_{k+l}) \in \mathbb{Z}^{k+l} $ and $f_{\bold{a}} \in \mathcal{H}_{rat} $ such that $T^{\bold{a}} f_{\bold{a}} = f_{\bold{a}}$, $|| f_{\bold{a}} - f || < \epsilon/2 $ and $\int f_{\bold{a}} \, d \mu = \mu(A)$.

 Choose $r$ large such that $a_i | r $ for all $i$.  Note that the set of $\{ \bold{d}_n^{(r)} \}$ has positive relative density in $D_1$ .
Consider 
$$ \frac{1}{N} \sum_{n=1}^N \mu(A \cap T^{-\bold{d}_n^{(r)}}A) =   
\frac{1}{N} \sum_{n=1}^N  \int f \, T^{\bold{d}_n^{(r)}} f \, d \mu   +   \frac{1}{N} \sum_{n=1}^N  \int g \,T^{\bold{d}_n^{(r)}} g \, d \mu. $$
For $f \in \mathcal{H}_{rat}$,
\begin{eqnarray*}
\int f \, T^{\bold{d}_n^{(r)}} f \, d \mu &=& \langle f_{\bold{a}}, f_{\bold{a}} \rangle + \langle f_{\bold{a}}, T^{\bold{d}_n^{(r)}} (f- f_{\bold{a}}) \rangle + \langle f-f_{\bold{a}}, T^{\bold{d}_n^{(r)}}f \rangle \\
& \geq& \mu^2(A) - \epsilon, 
\end{eqnarray*}
Also note that $(\bold{d}_n^{(r)}\cdot \bold{\gamma)}$ is u.d $\bmod 1$ for $\bold{\gamma} \notin (\mathbb{Q}/\mathbb{Z})^{k+l}$. Hence,
$$\frac{1}{N} \sum_{n=1}^N  \int g \,T^{\bold{d}_n^{(r)}} g \, d \mu = \int \frac{1}{N} \sum_{n=1}^N e(\bold{d}_n^{(r)}\cdot \bold{\gamma} ) \, d \nu(\bold{\gamma}) \rightarrow 0,$$
since $\nu(\mathbb{Q}/\mathbb{Z})^{k+l} = 0$ due to $g \in \mathcal{H}_{tot}$.
 Then, $$\frac{1}{N} \sum_{n=1}^N \mu(A \cap T^{-\bold{d}_n^{(r)}}A) \geq \mu(A)^2 - \epsilon.$$
By Proposition \ref{positive density}, $ \{ \bold{d} \in D_1 : \mu(A \cap T^{-\bold{d}}A ) \geq \mu^2(A) - \epsilon \} $
 has positive lower relative density in $D_1$.
 
 For (\ref{eqn9}), choose $\epsilon$ small such that $\mu^2(A) - \epsilon \geq \mu^2(A)/2$. Since $(\bold{d}_n^{(r)})$ has positive relative density, say $\alpha$, we have
$$\lim_{N \rightarrow \infty} \frac{1}{N}\sum_{n=1}^N \mu(A \cap T^{-\bold{d}_n}A) \geq \alpha \lim_{N \rightarrow \infty} \frac{1}{N}\sum_{n=1}^N \mu(A \cap T^{-\bold{d}_n^{(r)}}A) \geq \frac{\alpha}{2} \mu^2(A). $$

\end{proof}

Via Furstenberg's correspondence principle, one can deduce the following corollary. (See also proof of Corollary \ref{prop}.)

\begin{Corollary}
\label{semi ergodic cor}
Let $D_1$ and $D_2$ be as in Theorem \ref{sarkozy type}.
If $E \subset \mathbb{Z}^{k+l}$ with ${d^*}(E) > 0$, then for any $\epsilon > 0$
\begin{equation*}
 \{ \bold{d} \in D_i : d^*(E \cap E - \bold{d} ) \geq d^*(E)^2 - \epsilon \} 
 \end{equation*}
 has positive lower relative density in $D_i$ for $i=1, 2$.
 Furthermore,
 $$\liminf_{N \rightarrow \infty} \frac{\left| \{ p \leq N :\left(  (p - 1)^{\alpha_1}, \cdots , (p - 1)^{\alpha_k}, [(p - 1)^{\beta_1}], \cdots , [(p - 1)^{\beta_l}]  \right) \in E - E \} \right| }{\pi(N)} > 0.$$
  $$\liminf_{N \rightarrow \infty} \frac{\left| \{ p \leq N :\left(  (p + 1)^{\alpha_1}, \cdots , (p + 1)^{\alpha_k}, [(p + 1)^{\beta_1}], \cdots , [(p + 1)^{\beta_l}]  \right) \in E - E \} \right| }{\pi(N)} > 0.$$ 
\end{Corollary}

\renewcommand{\abstractname}{Acknowledgements}
\begin{abstract}
 We would like to thank Angelo Nasca for helpful remarks on the preliminary draft of this paper.  
 \end{abstract}

\end{document}